\documentclass[fleqn,leqno,11pt,dvips]{amsart}

\usepackage{a4wide}
\usepackage{url,graphicx,latexsym,amssymb,amsthm,amsmath,amsxtra,amsfonts}
 \usepackage[dvips]{color}

\theoremstyle{plain}

\newtheorem{theorem}{Theorem}

\newtheorem{proposition}{Proposition}

\newtheorem{remark}{Remark}

\def\du{\partial_u}
\def\dv{\partial_v}

\def\<{\langle}
\def\>{\rangle}

\def\<{ \left < }
\def\>{ \right > }

\def\R{\mathbb{R}}

\def\E{\mathbb{E}}
\def\H{\mathbb{H}}

\begin{document}

\title[]
{Surfaces in $\E^3$ making constant angle with Killing vector fields}
\author[M.I. Munteanu]{Marian Ioan Munteanu}
\author[A.I. Nistor]{Ana-Irina Nistor}

\date{ }

\email{(M.I.M.) marian.ioan.munteanu (at) gmail.com}
\email{(A.I.N.) ana.irina.nistor (at) gmail.com}

\begin{abstract}
In the present paper we classify curves and surfaces in Euclidean $3-$space which make constant angle with a certain Killing vector field.
Moreover, we characterize the catenoid and Dini's surface in terms of constant angle surfaces.
\end{abstract}


\date{\today}

\maketitle

\section{Introduction}
The geometry of curves and surfaces in $\E^3$ represented for many years a popular topic in the field of classical differential geometry.
Global and local properties were studied in several books and new invariants were defined for curves as well as for surfaces.

An old and important problem in differential geometry of curves and surfaces is that of Bertrand--Lancret--de Saint Venant saying that a curve in the
Euclidean $3-$space $\E^3$ is of constant slope, namely its tangent makes a constant angle with a fixed direction, if and only if the ratio of
torsion $\tau$ and curvature $\kappa$ is constant. These curves are also called \emph{generalized helices}. If both $\tau$ and $\kappa$ are
non-zero constants the curve is called \emph{circular helix}. Other interesting curves in $\E^3$ are \emph{slant helices}, characterized by the
property that the principal normals make a constant angle with a fixed direction.

The problem of Bertrand--Lancret--de Saint Venant was generalized for curves in other 3-dimensional manifolds -- in particular space forms or
Sasakian manifolds. Such a curve has the property that its tangent makes constant angle with a parallel vector field on the manifold or with a
Killing vector field respectively. For example, a curve $\gamma(s)$ in a $3-$dimensional space form is called a \emph{general helix} if there
exists a Killing vector field $V(s)$ with constant length along $\gamma$ and such that the angle between $V$ and $\gamma'$ is a non-zero
constant. See e.g. \cite{mn:Bar97}, \cite{mn:IL08}. Moreover, in \cite{mn:ABG04} it is shown that general helices in a $3-$dimensional space form
are extremal curves of a functional involving a linear combination of the curvature, the torsion (as functions) and a constant. On the other
hand, in a contact $3-$manifold, a curve is called a {\em slant curve} if the angle between its tangent and the Reeb vector field is constant
(see e.g. \cite{mn:CIL06}, \cite{mn:BB92}).
General helices, also called Lancret curves are used in many applications (e.g. \cite{mn:BM07}, \cite{mn:FHMS04}).

Bertrand studied curves in Euclidean space $\E^3$ whose principal normals are also the principal normals of another curve (called Bertrand mate).
Such a curve is nowadays called a Bertrand curve and it is characterized by a linear relation between its curvature and torsion.
Bertrand mates in $\E^3$ are particular examples of offset curves used in computer-aided geometric design (see e.g. \cite{mn:NM88}). Due to
previous characterization, one may say that Bertrand curves are strictly related with Weigarten surfaces. Possible applications of Weingarten
surfaces are found in \cite{mn:BG96}.

Passing from curves to surfaces, a natural question is to find which surfaces in $3-$dimensional space make constant angle with certain vector
fields. For example in \cite{mn:MN09} it is given a new approach to classify surfaces in $\E^3$ making a constant angle with a fixed direction.
These surfaces were called \emph{constant angle surfaces} and they appear in the study of liquid crystals and of layered fluids \cite{mn:CS07}.
Recently, the geometry of constant angle surfaces was developed in other $3-$dimensional spaces as ${\mathbb{S}}^2\times\R$ \cite{mn:DFVV07},
$\H^2\times\R$ \cite{mn:DM09}, the Heisenberg group $Nil_3$ \cite{mn:FMV10}, $Sol_3$ \cite{mn:LM10a}, Minkowski $3-$space \cite{mn:LM10}, a.s.o. On
the other hand, in analogy with the geometry of logarithmic spirals, in \cite{mn:Mun10} there were found all surfaces in $\E^3$ whose normals
make constant angle with the position vector. Particular surfaces with beautiful shapes were obtained.

The main part of the paper is devoted to the study of surfaces whose normals make constant angle $\theta$ with Killing vectors in $\E^3$. We
obtain the complete classification of these surfaces, which includes: halfplanes having $z-$axis as boundary, rotational surfaces around $z-$axis,
right cylinders over logarithmic spirals and Dini's surfaces. Regarding Dini's surfaces as a particular kind of helicoidal surfaces,
on this broader class of surfaces many studies were made, let us mention for example
\cite{mn:BK98}, \cite{mn:DD82}, \cite{mn:Rou88}, \cite{mn:Rou2000} and references therein.
Classification and characterization results were obtained having as starting point the fact that helicoidal surfaces are applicable upon
rotational surfaces.
For more recent developments and approaches on curves and surfaces, see for example \cite{mn:GAS}.

Different types of helicoidal surfaces were investigated also recently
in other ambient spaces as Minkowski space \cite{mn:LD}, or the product space $M^2\times \R$, \cite{mn:DL09}.

The study of these surfaces did not stop only in the field of classical differential geometry. Also interesting computer graphics results concerning helicoidal
surfaces were obtained in e.g. \cite{mn:HR91}.
We conclude the present paper with a new characterization of Dini's surfaces in terms of constant angle surfaces.

\section{Some results on curves in ${\mathbb{E}}^3$}

In the first part of this Section we bring to light new properties of planar and spatial curves involving the {\em angle of two curves}.
Since the result is local, let us consider two curves $\gamma$ and $\tilde\gamma$, without self-intersections, parameterized by the same
parameter $t\in I\subset\R$. For an arbitrary $t\in I$, we define the angle $\theta$ (at $t$) of $\gamma$ and $\tilde\gamma$ as the angle of the
tangent vectors $\gamma'(t)$ and $\tilde\gamma'(t)$ in the corresponding point. Interesting problems appear when the angle $\theta$ is constant.

Let us point our attention first on two planar curves which lie in the same plane. A particular case arises when one of the curves is a straight
line. If the other curve makes constant angle with this line, it is also a straight line.
A second case occurs when we are looking for planar curves $\gamma$ which make a constant angle with the unit circle ${\mathbb{S}}^1$. Take
\begin{equation}
\label{mn:gama} \gamma:I\subset\R \rightarrow\R^2,\ \gamma(s)=(\gamma_1(s),\ \gamma_2(s))
\end{equation} arc length parameterized and the unit circle
${\mathbb{S}}^1$ parameterized with the same parameter as $\gamma$,
\begin{equation}
\label{mn:circle} {\mathbb{S}} ^1(s)=(\cos\sigma(s),\ \sin\sigma(s))
\end{equation}
where $\sigma$ is a smooth function defined on $I$. Notice that, in general, the arc length parameters for the two curves are not the same.

Denote by $t_{{\mathbb{S}}^1}:=(-\sin\sigma(s),\ \cos\sigma(s))$ and $t_{\gamma}:=(\gamma_1^{'}(s),\ \gamma_2^{'}(s))$ the unit tangent vectors to ${\mathbb{S}}^1$
and to the curve $\gamma$, respectively. The angle between $\gamma$ and ${\mathbb{S}}^1$ is $\angle{(t_{{\mathbb{S}}^1},\ t_{\gamma})}$ and it will be a constant
$\theta\in[0,\pi)$.
It follows that
\begin{equation}
\label{mn:costheta} -\sin\sigma(s)\ \gamma_1^{'}(s) + \cos\sigma(s)\ \gamma_2^{'}(s)=\cos\theta.
\end{equation}
Straightforward computations yield the parametrization of the curve $\gamma$
\begin{eqnarray}
\label{mn:gammafinal1}
\gamma(s)&=& \left( \sin\theta\int\cos\sigma(s)d s -\cos\theta\int\sin\sigma(s)d s, \right.\\
& &\left. \cos\theta\int\cos\sigma(s)d s + \sin\theta\int\sin\sigma(s)d s \right) \nonumber
\end{eqnarray}
after a translation in $\R^2$.

If we denote $\displaystyle \gamma_0(s)=\left(-\int\sin\sigma(s)d s,\ \int\cos\sigma(s)d s \right)$, then the parametrization
\eqref{mn:gammafinal1} can be rewritten as
$$
\gamma(s)= \Big(\cos\theta- i \sin\theta \Big)\gamma_0(s).
$$
In order to give a geometric interpretation of the above formula, let us take a look at Figure~\ref{mn:curba_cerc} where the the angle
$\theta=\frac{\pi}{3}$, the function $\sigma(s)=s^2$ and $s \in [-\pi, \pi]$.

\begin{figure}[h]
\includegraphics[width=65mm]{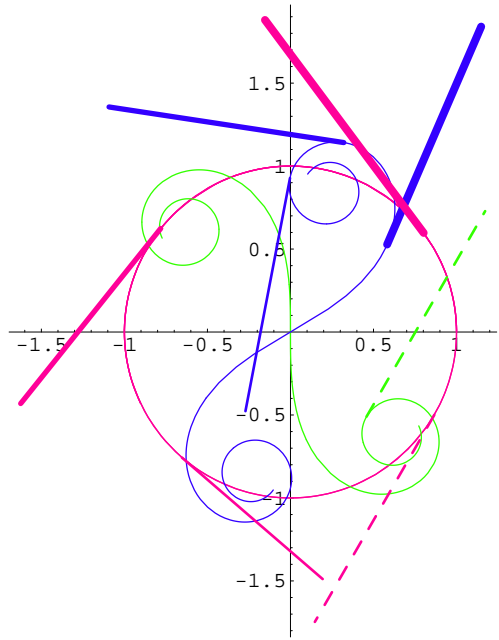}\\
\caption{Geometric interpretation.} \label{mn:curba_cerc}
\end{figure}
If the curve $\gamma_0$ (green) makes constant angle $\theta = 0$ with the circle ${\mathbb{S}}^1$ (magenta), then the curve $\gamma$ (blue) represents the
rotation with angle $\frac{\pi}{3}$ of $\gamma_0$ in clockwise direction. The line segments represent pairs of tangent vectors to the curve
$\gamma$
 and respectively to the circle ${\mathbb{S}}^1$.

In the same manner we investigate now the spatial curves. The first question that arises is to find spatial curves $\gamma$ which make a constant
angle with a straight line and a classical result tells us that $\gamma$ is a {\em helix}.
Without loss of generality the line can be taken to be (parallel with) one of the coordinate axes.
But this is an integral curve of a Killing vector vector field in ${\mathbb{R}}^3$.
Motivated by these remarks, a natural question springs up: {\em which curves make a constant angle with a Killing vector field in $\R^3$}?

Let us briefly recall that a vector field $V$ in $\R^3$ is Killing if and only if it satisfies the Killing equation:
$$
\langle {\stackrel{\circ}{\nabla}}_Y V, Z \rangle + \langle{\stackrel{\circ}{\nabla}}_Z V, Y \rangle = 0
$$
for any vector fields $Y$, $Z$ in $\R^3$, where ${\stackrel{\circ}{\nabla}}$ denotes the Euclidean connection on $\R^3$.
The set of solutions of this equation is
$$
\{\partial_x,\ \partial_y,\ \partial_z,\ -y\partial_x+x\partial_y,\ -z\partial_y+ y\partial_z,\ z\partial_x-x\partial_z\}
$$
and it gives a basis of Killing vector fields in $\R^3$. Here $x,y,z$ denote the global coordinates on $\R^3$ and $\R^3= {\rm span}\{\partial_x,
\partial_y, \partial_z\}$ regarded as a vector space.

As we have mentioned before, spatial curves which make a constant angle with any of the Killing vector fields $\partial_x,\
\partial_y,\ {\rm or}\ \partial_z$ are helices. Regarding the other three Killing vector fields, the problem is basically the same for each of them and
let us choose for example $V=-y\partial_x+x\partial_y$.

We study first planar curves $\gamma$ (in $xy-$plane) defined by \eqref{mn:gama} which make a constant angle $\theta$ with the
Killing vector field $V$.

Since the curve $\gamma$ is parameterized by arc length, the unit tangent vector is given by $t_{\gamma}=(\gamma_1',\ \gamma_2')$. Since
$V_{|_\gamma} =(-\gamma_2,\ \gamma_1)$, the condition $\angle{(t_\gamma,\ V|_{\gamma})}=\theta$ can be rewritten as
\begin{equation}
\label{mn:condV}
-\gamma_1'(s)\gamma_2(s)+\gamma_2'(s)\gamma_1(s)=\sqrt{\gamma_1(s)^2+\gamma_2(s)^2}~\cos\theta.
\end{equation}
A convenient way of handling $\gamma$ is to consider its polar parametrization,
$$
\gamma(s)=(r(s)\cos\phi(s),\ r(s)\sin\phi(s))
$$
and then equation \eqref{mn:condV} becomes
\begin{equation}
r(s)\phi'(s)=\cos\theta.
\end{equation}
Using now the fact that $\gamma$ is parameterized by arc length and combining it with the previous condition, we get after integration

\qquad if $\theta=0$: $r=r_0$ and $\phi=\frac s{r_0}+\phi_0$

\qquad if $\theta\neq0$: $r(s)= s \sin\theta+s_0$ and $\phi(s)=\cot\theta\ln(s \sin\theta+s_0)+\phi_0$

where $r_0>0$ and $s_0,\phi_0 \in \R$.

We state now the following theorem:
\begin{theorem}
A curve lying in the $xy$-plane and making constant angle with the Killing vector $V=-y\frac\partial{\partial x}+x\frac\partial{\partial y}$ is
\begin{itemize}
\item[a)] either the circle ${\mathcal{C}}(r_0)$ centered in the origin and of radius $r_0$
\item[b)] or a straight line passing through the origin
\item[c)] or the logarithmic spiral $r(\phi)=e^{\tan\theta~(\phi-\phi_0)}$.
\end{itemize}
\end{theorem}

\begin{remark} \rm This is not a surprising fact since the circle ${\mathcal{C}}(r_0)$ is an integral curve for $V$ and
the logarithmic spiral, known also as
\emph{the equiangular spiral}, is characterized by the property that the angle between its tangent and the radial direction at every point is
constant. Moreover, the position vector $p$ and the vector $V_p$ are orthogonal.
\end{remark}

\medskip

If $\gamma$ is a spatial curve making a constant angle with the Killing vector field $V$, let us consider $\gamma:I\rightarrow \R^3,\
\gamma(s)=(\gamma_1(s),\gamma_2(s),\gamma_3(s))$ given in cylindrical coordinates as:
\begin{equation}
\label{mn:gamacil} \gamma(s)=\Big(r(s)\cos\phi(s),\ r(s)\sin\phi(s),\ z(s)\Big)
\end{equation}
 where $s$ is again arc length parameter.

We have $t_\gamma=(\gamma_1',\ \gamma_2',\ \gamma_3')$ and $V|_\gamma=(-\gamma_2, \gamma_1,0)$. Similarly as in previous situation, the property
$\angle{\big(t_\gamma,\ V|_{\gamma}\big)}=\theta$ becomes
\begin{equation}
\label{mn:condVcil}
r(s)\phi'(s)=\cos\theta.
\end{equation}
Again, since $\gamma$ is parameterized by arc length, we get $r'(s)^2+z'(s)^2=\sin^2\theta$. Hence, there exists a function $\omega(s)$ such that
$$
r'(s)=\sin\theta\cos\omega(s),\ \ z'(s)=\sin\theta\sin\omega(s).
$$
Integrating once and combining with \eqref{mn:condVcil} we may formulate the following theorem:
\begin{theorem}
A curve $\gamma$ in the Euclidean space ${\mathbb{E}}^3$ which makes a constant angle with the Killing vector
$V=-y\frac\partial{\partial x}+x\frac\partial{\partial y}$
is given, in cylindrical coordinates, up to vertical translations and rotations around $z$-axis, by
$$
r(s)=r_0+\sin\theta\int\limits^s \cos\omega(\zeta) d\zeta,\ \ \ \ z(s)=\sin\theta\int\limits^s \sin\omega(\zeta) d\zeta,\ \ \ \
\phi(s)=\cos\theta\int\limits^s\frac{d\zeta}{r(\zeta)}
$$
where $r_0\in{\mathbb{R}}$ and $\omega$ is a smooth function on $I$.
\end{theorem}
Let us mention some particular cases for function $\omega$. (We will take $\theta$ different from $0$ and $\frac\pi2$.)
\begin{itemize}
\item
The function $\omega(s)=\omega_0$ is constant.
\begin{itemize}
\item If $\omega_0=0$, then $\gamma$ is a plane curve given by the logarithmic spiral $r(\phi)= \sin\theta e^{\tan \theta \phi}$.
\item If $\omega_0\neq0$, then  $\gamma$ lies on the cone $x^2+y^2-{\rm cotan }^2\omega_0 z^2=0$. Moreover,
the projection of $\gamma$ on the $xy-$plane is the logarithmic spiral $r(\phi)= \sin\theta\cos\omega_0 e^{\frac{\cos\omega_0}{{\rm cotan
\theta}}\phi}$. (See Figure~\ref{fspira1}.)
\end{itemize}
\item
If $\omega$ is an affine function, $\omega(s)= m s + n,\ m,n\in\mathbb{R}$, then parametrization \eqref{mn:gamacil} represents a curve on the
$2-$sphere centered in origin with radius $\frac{\sin\theta}{|m|}$. The projection of $\gamma$ on the $xy-$plane is given in polar coordinates by
$r(\phi)=\frac{\sin\theta}{m \cosh(\phi \tan\theta)}$~.(See Figure~\ref{fsfera1}.)
\item If $\omega(s)=\arccos(s)$, then the projection of $\gamma$ on the $xy-$plane is
$r(\phi)=\frac{2\sin\theta{\rm cotan}^2\theta}{\phi^2}$. (See Figure~\ref{farccos}.)
\end{itemize}
\begin{figure}[ht]
\begin{minipage}[b]{0.4\linewidth}
\includegraphics[scale=0.5]{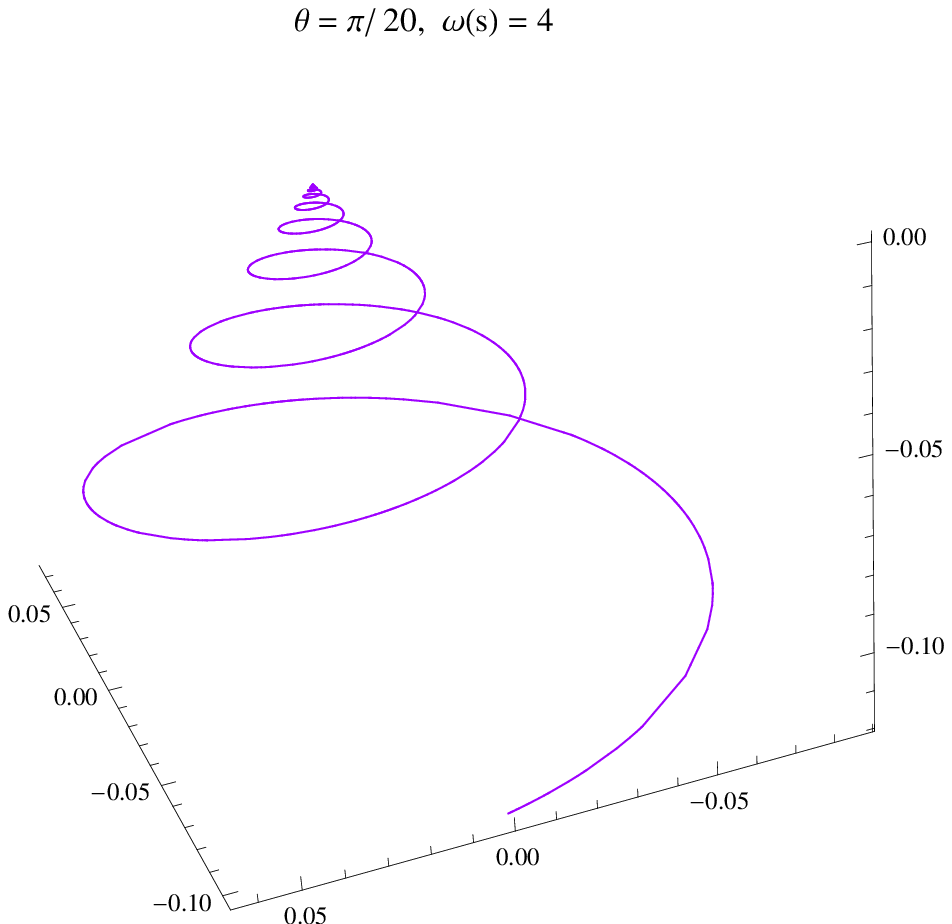}
\label{fig:figure1}
\end{minipage}
\begin{minipage}[b]{0.5\linewidth}
\centering
\includegraphics[scale=0.5]{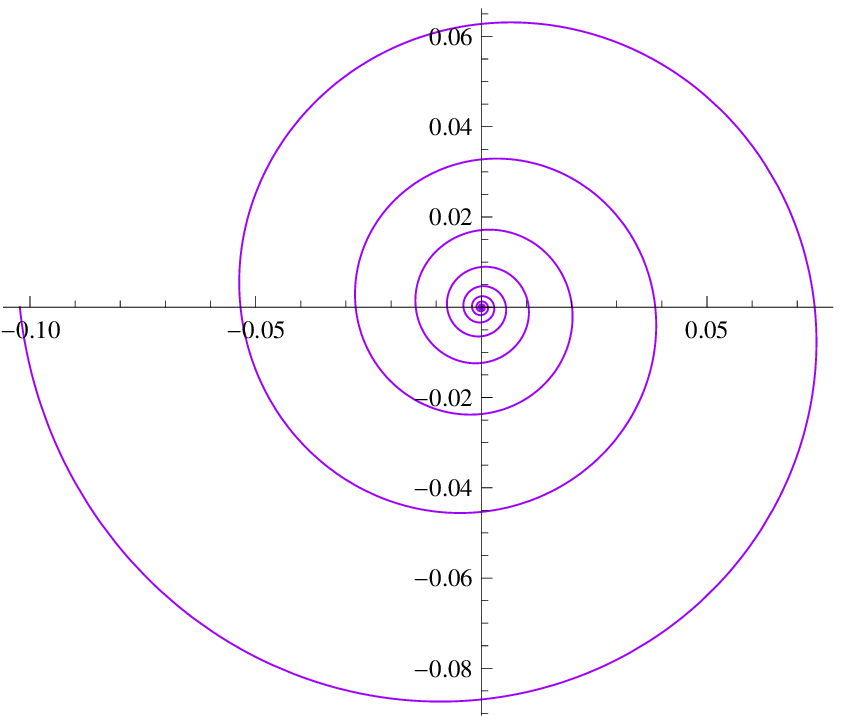}
\label{fig:figure2}
\end{minipage}
\caption{}
\label{fspira1}
\end{figure}

\begin{figure}
  \includegraphics[height=55mm]{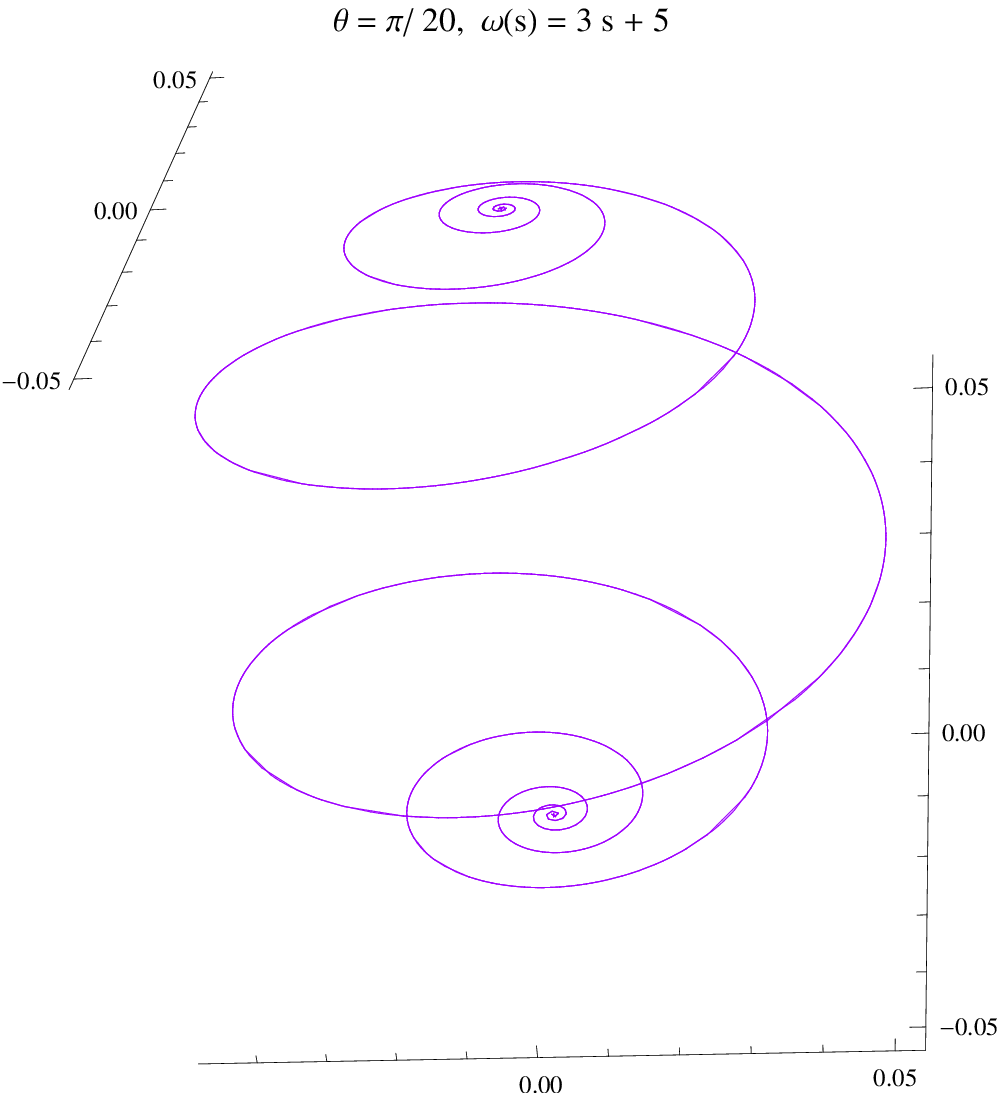} \hspace{5mm}
    \includegraphics[height=55mm]{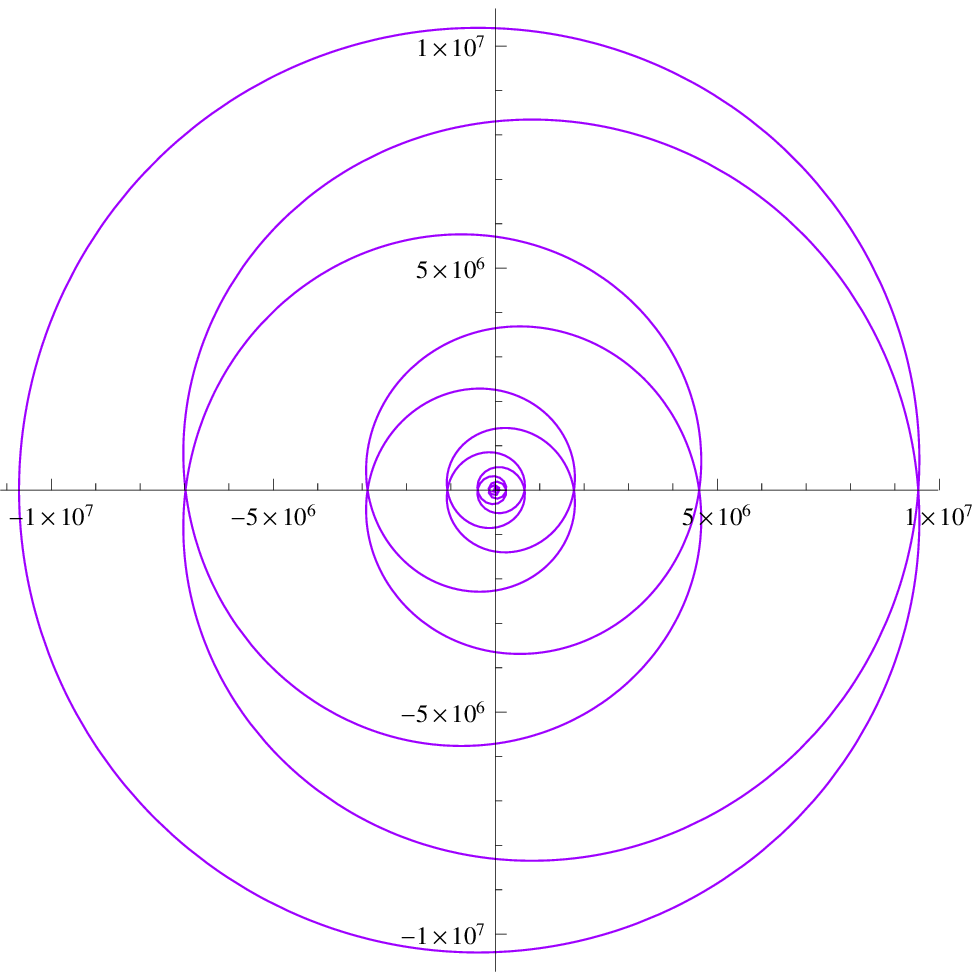}
  \caption{}
  \label{fsfera1}
\end{figure}

\begin{figure}
  \includegraphics[height=52mm]{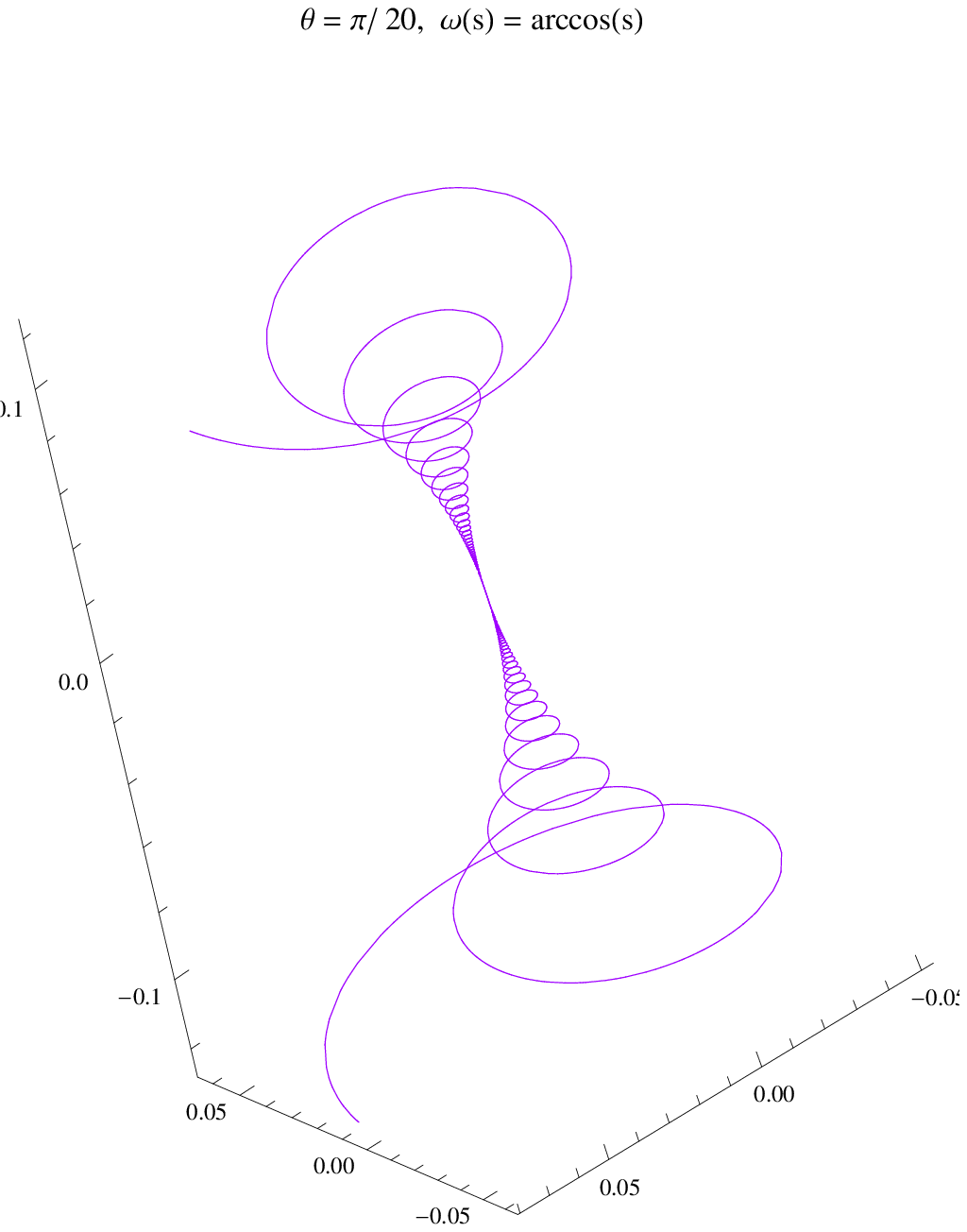} \hspace{5mm}
    \includegraphics[height=52mm]{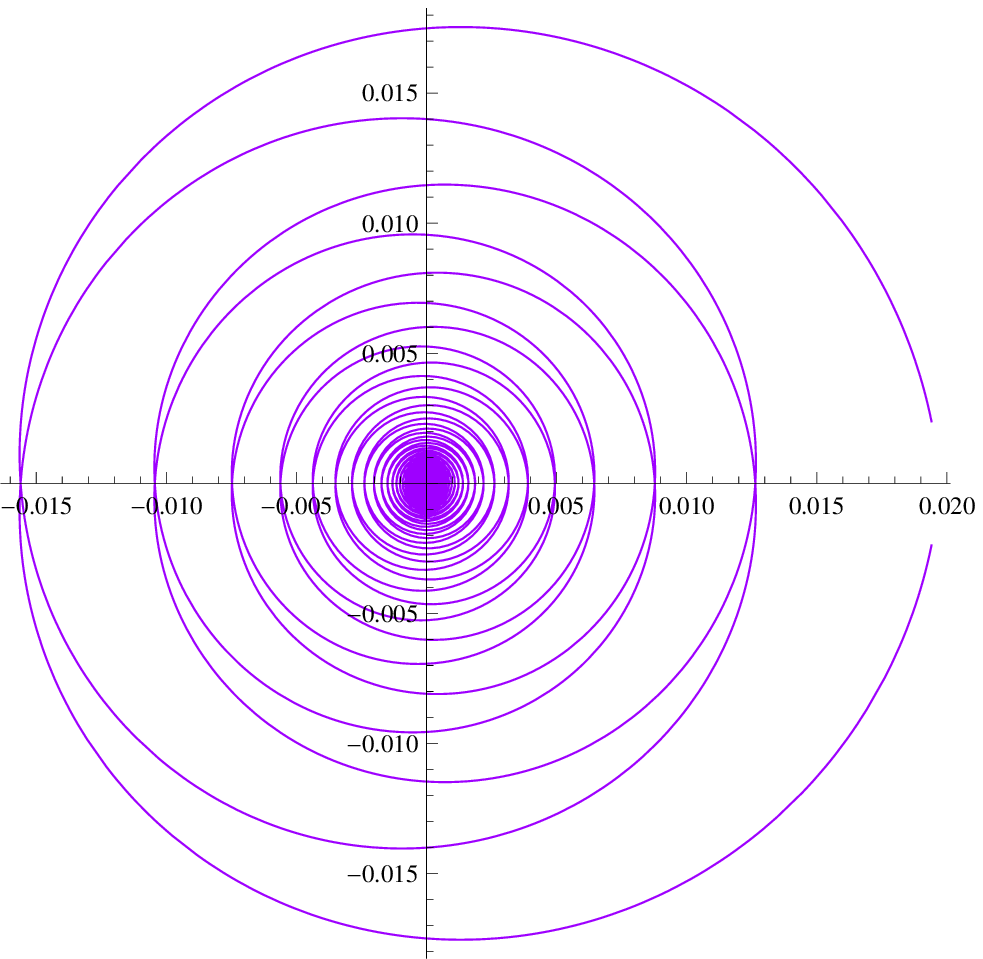}
  \caption{}
  \label{farccos}
\end{figure}

\newpage

\section{Classification of surfaces making constant angle with $V$}

In this section we are interested to find all {\em surfaces in Euclidean $3$-space which make a constant angle with the Killing vector field
$V$}.

Since the vector field $V= -y \partial_x + x \partial_y$ must be non-null, the surfaces lie in $\E^3\setminus Oz$.

A good motivation to study this problem is given by the particular case $\theta=\frac{\pi}{2}$. Let the surface $M$ be parameterized by
$F:D\subset\R^2 \rightarrow \E^3 \setminus \{Oz\}$ given by a graph
\begin{equation}
\label{mn:graph}
F(x,y)=(x,y, f(x,y)).
\end{equation}
Computing the normal to the surface and imposing for the vector field $V$ to be tangent to the surface in this case, namely to be orthogonal to
the unit normal $N$, one gets $y f_x-x f_y=0$. Hence, the surfaces for which $V$ makes a constant angle $\frac{\pi}{2}$ are given by
\eqref{mn:graph} with $f(x,y)=\mathbf{f}(\sqrt{x^2+y^2})$, where $\mathbf{f}$ is an arbitrary real valued function; they are rotational
surfaces.

Concerning the other particular case when the constant angle $\theta$ vanishes identically, we obtain halfplanes having $z-$axis as boundary.

We focus our attention now on the general case $\theta\neq 0,\frac{\pi}{2}$. First, let us fix the notations. Denote by $g$ the metric on $M$ and
by $\nabla$ the associated Levi Civita connection and let $\stackrel{\circ}{\nabla}$ be the flat connection in the ambient space. The Gauss and
Weingarten formulas are:

{\bf(G)}\qquad\qquad $\stackrel{\circ}{\nabla}_X Y=\nabla_X Y + h(X,Y)$

{\bf(W)}\qquad\qquad $\stackrel{\circ}{\nabla}_X N=-A X$

for every $X,Y$ tangent to $M$. Here $h$ is a symmetric $(1, 2)-$tensor field called the second fundamental form of the surface and $A$ is a
symmetric $(1, 1)-$tensor field called the shape operator associated to $N$ satisfying $\<h(X,Y),\ N\>=g(X,AY)$ for any $X,Y$ tangent to $M$.

Projecting $V$ onto the tangent plane to $M$ one has $V=T+\mu\cos\theta N$, where $T$ is the tangent part with $\|T\|=\mu \sin\theta$ and
$\mu=\|V\|$. At this point we may choose an orthonormal basis $\{e_1,e_2\}$ on the tangent plane to $M$ such that $e_1=\frac{T}{\|T\|}$ and
$e_2\perp e_1$. Now, we have that $V=\mu(\sin\theta e_1 +\cos\theta N)$.

For an arbitrary vector field $X$ in $\mathbb{E}^3$, we have
\begin{equation}
\label{mn:kX} \stackrel{\circ}{\nabla}_X V = k\times X
\end{equation}
where $k=(0,0,1)$ and $\times$ stands for the usual cross product in $\E^3$.

If $X$ is tangent to $M$, then
\begin{eqnarray}
\label{mn:XV}
\stackrel{\circ}{\nabla}_X V &=& X(\mu)\Big(\sin\theta e_1 +\cos\theta N\Big) \\
& & +~\mu\sin\theta \Big(\nabla_X e_1+ h(X,e_1)N \Big) - \mu\cos\theta AX. \nonumber
\end{eqnarray}
Since $e_1$ and $e_2$ are tangent to $M$, let us decompose $k\times e_1$ and $k\times e_2$ in the basis $\{e_1, e_2, N\}$ denoting the following
angles: $\angle(N,k):= \varphi$, $\angle(e_1,k):= \eta$ and $\angle(e_2,k):= \psi$. By standard calculus one gets up to sign
$$
\cos\varphi = -\sin\theta\sin\psi \ {\rm and}\ \cos\eta= \cos\theta\sin\psi.
$$
In these notations, we have
\begin{equation}
\label{mn:ke1e2} k\times e_1 = -\sin\theta \sin\psi~ e_2 - \cos\psi N,\ \  k\times e_2 = \sin\theta \sin\psi~ e_1 + \cos\theta\sin\psi N.
\end{equation}

Combining \eqref{mn:kX} and \eqref{mn:XV}, first for $X=e_1$ and then for $X=e_2$, together with \eqref{mn:ke1e2} we obtain
\begin{equation}
\label{mn:e1mu} e_1(\mu)=-\cos\theta\cos\psi
\end{equation}
\begin{equation}
\label{mn:e2mu} e_2(\mu)=\sin\psi.
\end{equation}
As a consequence, the shape operator associated to the second fundamental form is given by
\begin{equation}
\label{mn:h} A=\left(
  \begin{array}{cc}
    -\frac{\sin\theta\cos\psi}{\mu} & 0 \\
    0 & \lambda \\
  \end{array}
\right)
\end{equation}
where $\lambda$ is a smooth function on $M$. Hence $e_1$ and $e_2$ are principal directions on $M$.

In order to study the geometry of the surface $M$, we determine first the Levi Civita connection of $g$:
$$
\nabla_{e_1}e_1 =-\frac{\sin\psi}{\mu}~ e_2,\ \ \nabla_{e_1}e_2 =\frac{\sin\psi}{\mu}~ e_1,
$$
$$
 \nabla_{e_2}e_1 =\lambda~{\rm cotan}\theta~ e_2,\ \  \nabla_{e_2}e_2 =- \lambda~{\rm cotan}\theta~ e_1.
$$
Thus, the Lie bracket of $e_1$ and $e_2$ is given by
\begin{equation}
\label{mn:Lie_e1e2} [e_1, e_2]= \frac{\sin\psi}{\mu} e_1-\lambda~{\rm cotan}\theta~ e_2
\end{equation}
and consequently a compatibility condition is found computing $[e_1,e_2](\mu)$ in two ways and taking into account \eqref{mn:e1mu} and
\eqref{mn:e2mu}:
\begin{equation}
\label{mn:compatibility} -\cos\psi ~e_1(\psi)+\cos\theta \sin\psi~ e_2(\psi) = \frac{\cos\theta\sin\psi\cos\psi}{\mu}+\lambda~ {\rm
cotan}\theta\sin\psi.
\end{equation}
From now on we use cylindrical coordinates, such that the parametrization of the surface $M$ may be thought as
\begin{equation}
\label{mn:F} F:D\subset \R^2 \longrightarrow \E^3 \setminus\{Oz\},\ \ (u,v)\mapsto \Big(r(u,v), \phi(u,v),z(u,v)\Big).
\end{equation}
The Euclidean metric in $\E^3$ becomes a warped metric $\langle ~,~ \rangle = dr^2+dz^2+r^2d\phi^2$ and its Levi Civita connection is given by:
\begin{eqnarray}\label{mn:nabla_cil}
  \stackrel{\circ}{\nabla}_{\partial_r} \partial_r =0,\
  & \stackrel{\circ}{\nabla}_{\partial_r} \partial_\phi=\stackrel{\circ}{\nabla}_{\partial_\phi} \partial_r= \frac1 r~\partial_\phi,\
   & \stackrel{\circ}{\nabla}_{\partial_\phi} \partial_\phi= -r~ \partial_r, \\
  \stackrel{\circ}{\nabla}_{\partial_z} \partial_z=0,\  &
  \stackrel{\circ}{\nabla}_{\partial_z} \partial_\phi=\stackrel{\circ}{\nabla}_{\partial_\phi} \partial_z=0, \
  & \stackrel{\circ}{\nabla}_{\partial_r} \partial_z=\stackrel{\circ}{\nabla}_{\partial_z} \partial_r=0. \nonumber
\end{eqnarray}
The Killing vector field $V$ coincides with $\partial_\phi$.

The basis $\{e_1,e_2,N\}$ may be expressed in terms of the new coordinates as:
\begin{eqnarray}\label{mn:e1e2N}
  e_1 & = & -\cos\theta\cos\psi~\partial_r + \frac{\sin\theta}{\mu}~\partial_\phi +\cos\theta\sin\psi~\partial_z \nonumber \\
  e_2 & = & \sin\psi~\partial_r +\cos\psi~\partial_z \\
  N & = & \sin\theta\cos\psi ~\partial_r +\frac{\cos\theta}{\mu}~\partial_\phi - \sin \theta\sin\psi~\partial_z. \nonumber
\end{eqnarray}
In the sequel, since $\{e_1,e_2\}$ is an involutive system, in other words $[e_1,e_2]\in {\rm span}\{e_1,e_2\}$, one obtains:
\begin{equation}
\label{mn:crosete1e2} \frac{\cos\theta\sin\psi}{\mu}+e_1(\psi)=0.
\end{equation}
\begin{remark}\label{mn:rmk2} \rm
When $\psi$ is constant, it follows that $\sin\psi=0$ since $\theta\neq \frac{\pi}{2}$. Hence, from now on we will deal with $\psi\neq0$, while
the case $\psi=0$ will be treated separately.
\end{remark}
Writing \eqref{mn:crosete1e2} in an equivalent manner as $\frac{e_1(\psi)}{\sin\psi}=-\frac{\cos\theta}{r}$ and taking the derivative along $e_1$
we get
\begin{equation}
\label{mn:e1e1}
e_1(e_1(\psi))=0.
\end{equation}
Our intention is to find, locally, appropriate coordinates on the surface in order to write explicit embedding equations of $M$ in $\E^3$. Let us
choose a local coordinate $u$ such that $e_1=\frac{\partial}{\partial u}$. The second coordinate, denote it by $v$, will be defined later. In
this framework, equation \eqref{mn:e1e1} has the solution $\psi(u,v)=c(v) u + \tilde{c}$, where $\tilde{c}\in \R$ and $c\in C^\infty(M),\ c\neq
0\ \forall v$. Choosing $v$ such that $\frac{\partial\psi}{\partial v} = 0$ we may consider $c={\rm constant}$. Moreover, after a translation in
$u$ coordinate, one gets
\begin{equation}
\label{mn:psi} \psi = c u,\ c\in \R\setminus\{0\}.
\end{equation}
Substituting \eqref{mn:psi} in \eqref{mn:crosete1e2} easily follows
\begin{equation}
\label{mn:mu} \mu =-\frac{\cos\theta \sin(cu)}{c}\ .
\end{equation}
Now let us express $e_2$ in terms of $\partial_u$ and $\partial_v$ as
\begin{equation}
\label{mn:e2ab} e_2=a(u,v)\partial_u+b(u,v)\partial_v.
\end{equation}
Thus, $e_2(\mu)=-a(u,v)\cos\theta\cos(cu)$. On the other hand, replacing the expression of $\psi$ given by \eqref{mn:psi} in \eqref{mn:e2mu} we
get that $e_2(\mu)=\sin(cu)$. Hence,
\begin{equation}
\label{mn:a} a(u,v)=a(u)=-\frac{\tan(cu)}{\cos\theta}.
\end{equation}
Exploiting the compatibility condition \eqref{mn:compatibility} and combining it with \eqref{mn:a} we get
\begin{equation}
\label{mn:lambda} \lambda = -c\tan\theta\tan(cu).
\end{equation}
Consequently, $e_2=-\frac{\tan(cu)}{\cos\theta}\partial_u+b(u,v)\partial_v$. In terms of $u$ and $v$, from \eqref{mn:Lie_e1e2}, one obtains
$b(u,v)=\frac{b_0(v)}{\cos(cu)}$. After a homothetic transformation of the $v-$coordinate, one can take $b_0(v)=1$ and thus
\begin{equation}
\label{mn:b} b(u,v)=\frac{1}{\cos(cu)}.
\end{equation}
Replacing now \eqref{mn:a} and \eqref{mn:b} in \eqref{mn:e2ab} the expression of $e_2$ is explicitly found
\begin{equation}
\label{mn:e2} e_2=-\frac{\tan(cu)}{\cos\theta}\partial_u + \frac{1}{\cos(cu)}\partial_v.
\end{equation}
The Levi Civita connection may also be written in terms of the coordinates $u$ and $v$ in the following manner:
\begin{subequations}
\renewcommand{\theequation}{\theparentequation .\alph{equation}}
\label{mn:nablauv}
\begin{eqnarray}
\label{mn:nabla_uu}
\ \ \ \ \nabla_{\du} \du &=& \frac{c}{\cos\theta}\Big(-\frac{\tan(cu)}{\cos\theta}\du + \frac{1}{\cos(cu)}\dv\Big)\\[2mm]
\label{mn:nabla_uv}
\nabla_{\du} \dv &=&\nabla_{\dv} \du = c\tan^2\theta\sin(cu)\Big(-\frac{\tan(cu)}{\cos\theta}\du + \frac{1}{\cos(cu)}\dv \Big)  \\[2mm]
 \label{mn:nabla_vv}
\nabla_{\dv} \dv&=& -~c\tan^2\theta\sin(cu)\cos(cu)\du +\\
 & &\nonumber + ~\frac{c\tan^2\theta}{\cos\theta}\sin^2(cu)\Big(-\frac{\tan(cu)}{\cos\theta}\du + \frac{1}{\cos(cu)}\dv\Big).
\end{eqnarray}
\end{subequations}
At this point, using the expression of the shape operator \eqref{mn:h}, and the formulas \eqref{mn:psi}, \eqref{mn:mu} and respectively
\eqref{mn:lambda} we obtain the expression of the second fundamental form in terms of $u$ and $v$ as:
\begin{subequations}
\renewcommand{\theequation}{\theparentequation .\alph{equation}}
\begin{eqnarray}
\ \ \ \ h(\du,~ \du) &=& c~\tan\theta \ {\rm cotan} (cu) \\[2mm]
h(\du,~ \dv) &=&h(\dv,~ \du) = \frac{c~\tan\theta\cos(cu)}{\cos\theta} \\[2mm]
h(\dv,~ \dv) &=& c\tan^3\theta\sin(cu)\cos(cu).
\end{eqnarray}
\end{subequations}
Since the surface is given in cylindrical coordinates by the isometric immersion \eqref{mn:F}, and using the Euclidean connection
\eqref{mn:nabla_cil} one has
\begin{equation}
\label{mn:uu_1} \stackrel{\circ}{\nabla}_{\du} \du =(r_{uu}-r\phi_u^2,\ \phi_{uu}+2\frac{r_u}{r}\phi_u,\ z_{uu}).
\end{equation}
On the other hand, using Gauss formula {\bf (G)} one computes
\begin{eqnarray}
\label{mn:uu_2}
\stackrel{\circ}{\nabla}_{\du} \du &=&\Big(\frac{c~\sin^2(cu)+c~\sin^2\theta\cos^2(cu)}{\cos\theta\sin(cu)},\\
& &  \ \ \ -\frac{c^2\tan\theta\cos(cu)}{\sin^2(cu)},\ c~\cos\theta\cos(cu)\Big).\nonumber
\end{eqnarray}
Comparing the two previous expressions \eqref{mn:uu_1} and \eqref{mn:uu_2} we have the following PDEs:
\begin{subequations}
\renewcommand{\theequation}{\theparentequation .\alph{equation}}
\label{mn:Fuu}
\begin{eqnarray}
\label{mn:Fuu1}
\ \ \ \ r_{uu}-r\phi_u^2 &=& \frac{ c~\sin(cu)}{\cos\theta} + \frac{c~\sin^2\theta\cos^2(cu)}{\cos\theta\sin(cu)} \\[2mm]
\label{mn:Fuu2}
\phi_{uu}+2\frac{r_u}{r}~\phi_u & = & -\frac{c^2\tan\theta\cos(cu)}{\sin^2(cu)} \\[2mm]
 \label{mn:Fuu3}
 z_{uu} & = & c~\cos\theta\cos(cu).
\end{eqnarray}
\end{subequations}

Using the same technique also for the expressions of $\stackrel{\circ}{\nabla}_{\du} \dv$ and $\stackrel{\circ}{\nabla}_{\dv} \dv$, we get in
addition:
\begin{subequations}
\renewcommand{\theequation}{\theparentequation .\alph{equation}}
\label{mn:Fuv}
\begin{eqnarray}
\label{mn:Fuv1}
\ \ \ \ r_{uv}-r\phi_u\phi_v & = & c \tan^2\theta \\[2mm]
\label{mn:Fuv2}
\phi_{uv}+ \frac{1}{r}(r_u\phi_v + r_v\phi _u) & = & -\frac{c^2\tan\theta{\rm cotan}(cu)}{\cos\theta} \\[2mm]
 \label{mn:Fuv3}
 z_{uv} & = & 0
\end{eqnarray}
\end{subequations}
respectively
\begin{subequations}
\renewcommand{\theequation}{\theparentequation .\alph{equation}}
\label{mn:Fvv}
\begin{eqnarray}
\label{mn:Fvv1}
\ \ \ \ r_{vv}-r\phi_v^2 &=& \frac{ c~\tan^2\theta \sin(cu)}{\cos\theta} \\[2mm]
\label {mn:Fvv2}
\phi_{vv}+2\frac{r_v}{r}~\phi_v & = & 0 \\[2mm]
\label{mn:Fvv3}
 z_{vv} & = & 0.
\end{eqnarray}
\end{subequations}
First, let us note that the (cylindrical) coordinate $r$ has the same meaning as the function $\mu$, namely $r(u,v)=\mu(u)$ is given by
\eqref{mn:mu}. As a consequence, $r_v=0$. Moreover, from the choice of $u$ and $v$, together with \eqref{mn:e1e2N},
we immediately obtain the third component of the parametrization
\begin{equation}
\label{mn:z} z(u,v)=v-\frac{\cos\theta\cos(cu)}{c}
\end{equation}
after a translation along $z-$axis. Furthermore, straightforward computations yield
\begin{equation}
\label{mn:phi} \phi(u,v)=-\frac{cv~\tan\theta}{\cos\theta}- \tan\theta \log\Big(\tan\big(\frac{c u }{2}\big)\Big)\ .
\end{equation}
Notice that all equations \eqref{mn:Fuu}, \eqref{mn:Fuv} and \eqref{mn:Fvv} are verified.

Hence, when $\psi\neq 0$, combining \eqref{mn:mu}, \eqref{mn:z} and \eqref{mn:phi}, we get the following
parametrization in cylindrical coordinates:
{\small
\begin{equation}
\label{mn:helicoidal}
F(u,v)= \left( \frac{\cos\theta\sin(c u)}{c},\ -\frac{cv\tan\theta}{\cos\theta}
-\tan\theta\log\big(\tan(\frac{cu}{2})\big),
 v - \frac{\cos\theta \cos(c u)}{c} \right)
\end{equation}}
where $c$ is a nonzero real constant.

\medskip

Let us return now to the case $\psi=0$. Following the same steps as in the general case, we get the next parametrization in cylindrical
coordinates
\begin{equation}
\label{mn:para_psi0} F(u,z)=(u\cos\theta,\ \log(c u ^{-\tan\theta}),\ z).
\end{equation}
These surfaces are right cylinders over logarithmic spirals.

At this point one can state the main result of this section.

\begin{theorem}
\label{mn:thm_clasif}
Let $M$ be a surface isometrically immersed in $\E^3\setminus\{ Oz\}$ and let
 $V=-y\partial_x+x\partial_y$ be a Killing vector field.
Then $M$ makes a constant angle $\theta$ with $V$ if and only if it is one of the following surfaces:
\begin{enumerate}
\item[(i)] either a halfplane with $z-$axis as boundary
\item[(ii)] or a rotational surface around $z-$axis
\item[(iii)] or a right cylinder over a logarithmic spiral given by \eqref{mn:para_psi0}
\item[(iv)] or, finally, the Dini's surface defined in cylindrical coordinates by \eqref{mn:helicoidal}.
\end{enumerate}
\end{theorem}
\begin{proof}
The direct implication is the conclusion of the previous computations, so we briefly sketch it.

We analyze different cases for the constant angle $\theta$:
\begin{itemize}
\item $\theta=0$~: the halfplanes passing through $z-$axis as in case (i) are obtained;
\item $\theta=\frac{\pi}{2}$~: we get rotational surfaces, case (ii);
\item $\theta\neq 0, \frac{\pi}{2}$~: we distinguish two cases, accordingly to Remark~\ref{mn:rmk2}.
We obtain the cylinders over logarithmic spirals parameterized by \eqref{mn:para_psi0} in case (iii) and the surfaces parameterized by
\eqref{mn:helicoidal} in case {\rm (iv)}.
\end{itemize}
The converse part follows by direct computations showing that the surfaces in each of the cases (i)--(iv) make a constant angle $\theta$
with the Killing vector field $V$.
\end{proof}

Willing to give more information about the geometry of Dini's surface described in case (iv) of Theorem~\ref{mn:thm_clasif},
we note the following
\begin{remark}\rm
The surface given by \eqref{mn:helicoidal} is a helicoidal surface with axis  $Oz$, namely it can be parametrized as
$$
F(r,\phi)=\big(r\cos\phi,\ r\sin\phi,\ h\phi+ \Lambda(r)\big),
$$
where $(\Lambda\circ r)(u)=-\frac{\cos\theta}{c}~\big(\log(\tan(\frac{c u}{2}))+ \cos(c u)\big)$.
The pitch is $h=-\frac{\cos\theta}{c \tan\theta}$.
\end{remark}
Recall a classical construction \cite[\S15.7]{mn:GAS}: A helicoidal surface, known also as a generalized helicoid, is always related to
a rotational surface (when the pitch $h$ vanishes identically). Twisting the pseudosphere
of radius $\frac{\cos\theta}{c}$, one obtains the Dini's surfaces \eqref{mn:helicoidal}.

The next result makes a connection with the previous Section.
\begin{proposition}
\label{mn:prop1}
The parametric curves of surfaces parameterized by \eqref{mn:helicoidal} are circular helices and spherical
curves.
\end{proposition}
\begin{proof}
The $v-$parameter curves ($u=u_0\in{\mathbb{R}}$) are circular helices with the same pitch $-2\pi\frac{\cos\theta}{c\tan\theta}$ and with
radius $\frac{\cos\theta\sin(c u_0)}{c}$, depending on $u_0$.
On the other hand, the $u-$parameter curves ($v=v_0\in{\mathbb{R}}$) are unit speed
spherical curves and they lie on 2-spheres centered in $(0,0,v_0)$ and with the same radius $\frac{\cos\theta}{c}$.
\end{proof}
\begin{remark}\rm
Looking backward to Section 2, the $u-$parameter curves make the constant angle $\frac{\pi}{2}-\theta$ with the Killing vector field $V$
and the affine function $\omega$ is given by $\omega(s)=cs,\ c\in \R\setminus\{0\}$.
\end{remark}
We conclude this paper with the following proposition which shows that surfaces making constant angle with $V$ define
a particular class of Weingarten surfaces.
\begin{proposition}
Let $M$ be a surface isometrically immersed in $\E^3$ which makes a constant angle with the Killing vector field $V=-y\partial_x+x\partial_y$.
Then,
\begin{enumerate}
\item $M$ is totally geodesic if and only if it is a vertical plane with the boundary $z-$axis;
\item $M$ is minimal not totally geodesic if and only if it is a catenoid around $z-$axis;
\item $M$ is flat if and only if it is a vertical plane with the boundary $z-$axis, a flat rotational surface or a right cylinder over a logarithmic
spiral;
\item the Dini's surface from case {\rm (iv)} of the classification theorem has constant negative Gaussian curvature, $K=-c^2\tan^2\theta$.
\end{enumerate}
\end{proposition}


\begin{thebibliography}{99}

\bibitem{mn:ABG04} J.~Arroyo, M.~Barros, O.~J.~Garay,
        \emph{Models of relativistic particle with curvature and torsion revisited},
        Gen. Relativity Gravitation, 36 (2004) 6, 1441--1451.

\bibitem{mn:BB92} C.~Baikoussis, D.~E.~Blair,
        \emph{Integral surfaces of Sasakian space forms},
        J. Geom., 43 (1992), 30--40.

\bibitem{mn:BK98} C.~Baikoussis, T.~Koufogiorgos,
        \emph{Helicoidal surfaces with prescribed mean or Gaussian curvature},
        J. Geom., 63 (1998), 25--29.


\bibitem{mn:Bar97} M.~Barros,
        \emph{General helices and a theorem of Lancret},
        Proc. Amer. Math. Soc., 125 (1997), 1503--1509.

\bibitem{mn:BM07} J.~V.~Beltran, J.~Monterde,
        \emph{A characterization of quintic helices},
        J. Comput. Appl. Math., 206 (2007) 1, 116--121.


\bibitem{mn:BG96} B.~van Brunt, K.~Grant,
                \emph{Potential applications of Weingarten surfaces in CAGD, Part I: Weingarten surfaces and surface shape investigation},
                CAGD, 13 (1996), 569--582.

\bibitem{mn:CS07} P.~Cermelli, A.~J.~Di Scala,
        \emph{Constant-angle surfaces in liquid crystals},
        Philosophical Magazine, 87 (2007) 12, 1871--1888.

\bibitem{mn:CIL06} J.~T.~Cho, J.-I.~Inoguchi, J.~E.~Lee,
    \emph{Slant curves in Sasakian space forms},
    Bull. Austral. Math. Soc., 74 (2006) 3, 359--367.


\bibitem{mn:DL09} M.~Dajczer, J.~H.~de Lira,
        \emph{Helicoidal graphs with prescribed mean curvature},
         Proc. AMS, 137 (2009) 7, 2441--2444.

\bibitem{mn:DFVV07} F.~Dillen, J.~Fastenakels, J.~Van der Veken, L.~Vrancken,
        \emph{Constant Angle Surfaces in ${\mathbb{S}}^2\times{\mathbb{R}}$},
        Monatsh. Math., 152 (2007) 2, 89--96.

\bibitem{mn:DM09} F.~Dillen, M.~I.~Munteanu,
        \emph{Constant Angle Surfaces in ${\mathbb{H}}^2\times{\mathbb{R}}$},
         Bull. Braz. Math. Soc., 40 (2009) 1, 85--97.

\bibitem{mn:DD82} M.~P.~ Do Carmo, M.~ Dajczer,
        \emph{Helicoidal surfaces with constant mean curvature},
         T\^ohoku Math. J., 34 (1982), 425--435.

\bibitem{mn:FHMS04} R.~T.~Farouki, C.~Y.~Han, C.~Manni, A.~Sestini,
        \emph{Characterization and construction of helical polynomial space curves},
        J. Comput. Appl. Math., 162 (2004) 2, 365--392.

\bibitem{mn:FMV10} J.~Fastenakels, M.~I.~Munteanu, J.~Van der Veken,
        \emph{Constant Angle Surfaces in the Heisenberg group},
        to appear in Acta Mathematica Sinica.

\bibitem{mn:GAS} A.~Gray, E.~Abbena, S.~Salamon ,
        \emph{Modern Differential Geometry of Curves and Surfaces with Mathematica},
        Third edition, Studies in Adv. Math., Chapman \& Hall/CRC, 2006.

\bibitem{mn:HR91} L.~Hitt, I.~M.~Roussos,
    \emph{Computer graphics of helicoidal surfaces with constant mean curvature},
      An. Acad. Bras. Ci., 63 (1991), 211--228.

\bibitem{mn:IL08} J.-I.~Inoguchi, S.~Lee,
        \emph{Null Curves in Minkowski $3$-space},
        Int. Electron. J. Geom., 1 (2008) 2, 40--83.


\bibitem{mn:LD} R.~L\'opez, E.~Demir,
            \emph{Helicoidal surfaces in Minkowski space with constant mean curvature and constant Gauss curvature},
                arXiv:1006.2345v2 [math.DG].

\bibitem{mn:LM10} R.~L\'opez, M.~I.~Munteanu,
        \emph{Constant Angle Surfaces in Minkowski space},
         to appear in Bull. Belg. Math. Soc. Simon Stevin.

\bibitem{mn:LM10a} R.~L\'opez, M.~I.~Munteanu,
        \emph{On the Geometry of Constant Angle Surfaces in $Sol_3$},
        arXiv:1004.3889v1  [math.DG] 2010.

\bibitem{mn:Mun10} M.~I.~Munteanu,
        \emph{From Golden Spirals to Constant Slope Surfaces},
        J. Math. Phys., 51 (2010) 7, 073507:1--9.

\bibitem{mn:MN09}  M.~I.~Munteanu, A.~I.~Nistor,
        \emph{A new approach on Constant Angle Surfaces in ${\mathbb{E}}^3$},
        Turk. J. Math., 33 (2009), 169--178.


\bibitem{mn:NM88} A.~W.~Nutbourne, R.~R.~Martin,
    \emph{Differential Geometry Applied to the Design of Curves and Surfaces},
    Eellis Horwood, Chichester, UK, 1988.

\bibitem{mn:Rou88} I.~M.~Roussos,
    \emph{A geometric characterization of helicoidal surfaces of constant mean curvature},
      Publ. Inst. Math., N.S. 43 (1988) 57, 137--142.

\bibitem{mn:Rou2000} I.~M.~Roussos,
    \emph{Surfaces in $\E^3$ invariant under a one parameter group of isometries of $\E^3$},
      An. Acad. Bras. Ci., 72 (2000) 2, 125--159.


\end{thebibliography}
\end{document}